\documentclass[leqno,12pt]{article} 
\usepackage{amsmath, amssymb}
\usepackage{amsthm} 
\theoremstyle{plain} 
\newtheorem{theorem}{\indent\sc Theorem}[section]
\newtheorem{lemma}[theorem]{\indent\sc Lemma}

\newtheorem{proposition}[theorem]{\indent\sc Proposition}

\theoremstyle{definition} 

\newtheorem{remark}[theorem]{\indent\sc Remark}

\makeatletter
\def\address#1#2{\begingroup
\noindent\parbox[t]{7.8cm}{%
\small{\scshape\ignorespaces#1}\par\vskip1ex
\noindent\small{\itshape E-mail address}%
\/: #2\par\vskip4ex}\hfill%
\endgroup}%
\makeatother
\title{Classical discrete operators on variable $\ell^{p(\cdot)}(\mathbb{Z})$ spaces} 
\author{
\textsc{Pablo Rocha} 
}
\date{} 
\begin{document}

\maketitle


\begin{abstract}
We show, by applying discrete weighted norm inequalities and the Rubio de Francia algorithm, that the discrete Hilbert transform and discrete Riesz potential are bounded on variable $\ell^{p(\cdot)}(\mathbb{Z})$ spaces whenever the discrete Hardy-Littlewood maximal is bounded on $\ell^{(p(\cdot)/r)'}(\mathbb{Z})$, for some $r > 1$. We also obtain vector-valued inequalities for the discrete fractional maximal operator.
\end{abstract}

\section{Introduction}   

Given a sequence $p(\cdot) : \mathbb{Z} \to [1, \infty)$, we define the variable $\ell^{p(\cdot)} = \ell^{p(\cdot)}(\mathbb{Z})$ space to be
the set of all complex sequences $a = \{ a(i)  \}_{i \in \mathbb{Z}}$ such that
\[
\| a \|_{\ell^{p(\cdot)}} = \inf \left\{ \lambda > 0 : \sum_{i \in \mathbb{Z}} |a(i)/\lambda|^{p(i)} \leq 1 \right\} < \infty.
\]
It is well known that the couple $(\ell^{p(\cdot)},  \| \cdot \|_{\ell^{p(\cdot)}})$ results a Banach space. W. Orlicz \cite{Orlicz} and H. Nakano \cite{Nakano} were the first in studying these kind of spaces (see also \cite{Edmunds}, \cite{Nekvinda}, \cite{Nekvinda2}). 
P. H\"ast\"o \cite{Hasto} showed that the variable $\ell^{p(\cdot)}$ spaces have applications to the study of operators on variable Lebesgue spaces in $\mathbb{R}^n$.

Let $0 \leq \alpha < 1$, for a sequence $a = \{ a(i) \}_{i \in \mathbb{Z}}$, define the fractional Hardy-Littlewood maximal sequence 
$M_{\alpha}a$ by
\begin{equation} \label{Fract max}
(M_{\alpha}a)(j) = \sup \frac{1}{(n-m+1)^{1-\alpha}} \sum_{i=m}^n |a(i)|, \,\,\,\, j \in \mathbb{Z},
\end{equation}
where the supremum is taken over all $m, n \in \mathbb{Z}$ such that $m \leq j \leq n$. When $\alpha=0$, we have that $M_0 = M$, where $M$ is the discrete Hardy-Littlewood maximal operator. Recently, A. Swarup and A. Alphonse \cite{Swarup} proved that the discrete fractional maximal $M_{\alpha}$ is a bounded operator $\ell^{p(\cdot)} \to \ell^{q(\cdot)}$, for $\frac{1}{p(\cdot)} - \frac{1}{q(\cdot)} = \alpha$,  under the assumption that
\[
1 < p_{-}:= \inf \{ p(i) : i\in \mathbb{Z} \} \leq p_{+}:= \sup \{ p(i) : i\in \mathbb{Z} \} < \alpha^{-1} \,\,\, \text{and} \,\,\, 
p(\cdot) \in LH_{\infty}(\mathbb{Z}).
\]
We say that $p(\cdot) \in LH_{\infty}(\mathbb{Z})$, if there exist positive real 
constants $p_{\infty}$ and $C_{\infty}$ such that
\[
|p(i) - p_{\infty}| \leq \frac{C_{\infty}}{\log(e + |i|)}, \,\,\,\, \text{for all} \,\, i \in \mathbb{Z}.
\]
This property is known as the log-H\"older continuity at infinity for an exponent $p(\cdot)$. We set
\[
\mathcal{P} = \{ p(\cdot) : \mathbb{Z} \to [1, \infty) : 1 < p_{-} \leq p_{+} < \infty \},
\]
and
\[
\mathcal{B} = \{ p(\cdot) \in \mathcal{P} : M \,\, \text{is bounded on} \,\, \ell^{p(\cdot)} \}.
\]
Then, $LH_{\infty} \cap \mathcal{P} \subset \mathcal{B}$. Such inclusion also immediately follows from a result due to 
A. Nekvinda; indeed, in \cite{Nekvinda} he gave a necessary and sufficient condition $\widetilde{\mathcal{P}}$ for sequences 
$p(\cdot)$ to guarantee an existence of a real number $r \geq 1$ such that norms in $\ell^{p(\cdot)}$ and in $\ell^r$ are equivalent. 
By \cite[Corollary 4.1.9]{Diening}, we have that if $p(\cdot) \in LH_{\infty} \cap \mathcal{P}$, then 
$\ell^{p(\cdot)} \equiv \ell^{p_\infty}$ with equivalent norms (i.e.: $LH_{\infty} \cap \mathcal{P} \subset \widetilde{\mathcal{P}}$). 
From this equivalence and \cite[Theorem 2.3 and Proposition 2.4]{Rocha}, we immediately obtain that $M_{\alpha}$ is a bounded 
operator $\ell^{p(\cdot)} \to \ell^{q(\cdot)}$, where $\frac{1}{q(\cdot)} = \frac{1}{p(\cdot)} - \alpha$ and $0 \leq \alpha < 1$, when 
$p(\cdot) \in LH_{\infty} \cap \mathcal{P}$ and $p_{+} < \alpha^{-1}$. In particular, we have that $LH_{\infty} \cap \mathcal{P} \subset \mathcal{B}$. 

On the other hand, by \cite[Theorem 5.10]{Nekvinda2}, one has that $\mathcal{B} \setminus \widetilde{\mathcal{P}} \neq \emptyset$.

Given an exponent sequence $p(\cdot) : \mathbb{Z} \to [1, \infty)$, its conjugate exponent $p'(\cdot)$ is defined by 
$\frac{1}{p(i)} + \frac{1}{p'(i)} = 1$, for all $i \in \mathbb{Z}$. 

We point out that in our main results, we will only assume that the exponent $(p(\cdot)/r)' \in \mathcal{B}$ for some $1 < r < p_{-}$.

\begin{remark} \label{lps}
One can see that $\ell^{p(\cdot)} = \{ a : \sum_{i \in \mathbb{Z}} |a(i)|^{p(i)} < \infty \}$ and 
$\ell^{p(\cdot)} \subset \ell^{p_{+}}$, when $p_{+} < \infty$.
\end{remark}

Next, we introduce two classical discrete operators. Given a sequence $a = \{ a(i) \}_{i \in \mathbb{Z}}$, define the discrete Hilbert transform by
\begin{equation} \label{Hilbert op}
(Ha)(j) = \sum_{i \neq j} \frac{a(i)}{i-j}, \,\,\,\, j \in \mathbb{Z}.
\end{equation}
In \cite{Riesz}, M. Riesz proved that $H$ is a bounded operator $\ell^r \to \ell^r$ for $1 < r < \infty$. Thus, by Remark \ref{lps}, 
the operator $H$ is well defined on $\ell^{p(\cdot)}$ for every $p(\cdot) \in \mathcal{P}$. 

For $0 < \alpha < 1$, the discrete Riesz potential is defined by
\begin{equation} \label{Riesz op}
(I_{\alpha}a)(j) = \sum_{i \neq j} \frac{a(i)}{|i-j|^{1-\alpha}}, \,\,\,\, j \in \mathbb{Z}.
\end{equation}
Since $I_{\alpha}$ is a bounded operator $\ell^r \to \ell^s$ for $1 < r < \alpha^{-1}$ and $\frac{1}{s} = \frac{1}{r} - \alpha$ 
(see \cite{Hardy}, p. 288), by Remark \ref{lps}, it follows that the operator $I_{\alpha}$ is well defined on $\ell^{p(\cdot)}$ for every 
$p(\cdot) \in \mathcal{P}$, with $p_{+} < \alpha^{-1}$.

The purpose of this note, which was inspired by the work \cite{Uribe}, is to prove the following three results.

\begin{theorem} \label{thm 1}
Let $H$ be the discrete Hilbert transform given by (\ref{Hilbert op}) and $1 < r < p_{-}$. If $(p(\cdot)/r)' \in \mathcal{B}$, then there exists a positive constant $C$ such that
\[
\| Ha \|_{\ell^{p(\cdot)}} \leq C \| a \|_{\ell^{p(\cdot)}},
\]
for all $a \in \ell^{p(\cdot)}$.
\end{theorem}

\begin{theorem} \label{thm 2}
Let $I_{\alpha}$ be the discrete Riesz potential given by (\ref{Riesz op}) and $\frac{1}{1-\alpha} < s < q_{-}$. If  
$(q(\cdot)/s)' \in \mathcal{B}$ and $\frac{1}{p(\cdot)} = \frac{1}{q(\cdot)} + \alpha$, then there exists a positive constant $C$ such that
\[
\| I_{\alpha}a \|_{\ell^{q(\cdot)}} \leq C \| a \|_{\ell^{p(\cdot)}},
\]
for all $a \in \ell^{p(\cdot)}$.
\end{theorem}

\begin{theorem} \label{thm 3}
Given $0 \leq \alpha <1$, let $M_{\alpha}$ be the discrete fractional maximal given by (\ref{Fract max}) and 
$\frac{1}{1-\alpha} < s < q_{-}$. If $(q(\cdot)/s)' \in \mathcal{B}$, $\frac{1}{p(\cdot)} = \frac{1}{q(\cdot)} + \alpha$ and 
$\theta \in (1, \infty)$, then there exists a positive constant $C$ such that
\begin{equation} \label{fract max ineq}
\left\| \left\{ \sum_{k=1}^{\infty} (M_{\alpha}a_k)^{\theta} \right\}^{1/\theta} \right\|_{\ell^{q(\cdot)}} \leq C 
\left\| \left\{ \sum_{k=1}^{\infty} |a_k|^{\theta} \right\}^{1/\theta} \right\|_{\ell^{p(\cdot)}},
\end{equation}
for all sequences of functions $\{ a_k : \mathbb{Z} \to \mathbb{C} : k \in \mathbb{N}  \}$ with finite support.
\end{theorem}

This paper is organized as follows. Section 2 presents some facts about the $\ell^{p(\cdot)}$ spaces and it also introduces the discrete weights. In Section 3, we prove our main results by applying discrete weighted norm inequalities and the Rubio de Francia algorithm.

\

{\bf Notation.} Given a sequence $a = \{ a(i) \}$ and a real number $s > 0$, we put $|a|^s = \{ |a(i)|^s \}$. The symbol $A \approx B$ stands for the inequality $c_1 B \leq A \leq c_2 B$, where $c_1$ and $c_2$ are positive constants. Throughout this paper, $C$ will denote a positive real constant not necessarily the same at each occurrence.

\section{Preliminaries}

In this section, we give some basic results about the variable $\ell^{p(\cdot)}$ spaces and also introduce the discrete Muckenhoupt classes.

\begin{lemma} \label{potencia r}
Let $p(\cdot) \in \mathcal{P}$, then for every $r \in (0, p_{-})$
\[
\| a \|_{\ell^{p(\cdot)}}^{r} = \| |a|^{r} \|_{\ell^{p(\cdot)/r}}.
\]
\end{lemma}

\begin{proof} It follows from the definition of the $\ell^{p(\cdot)}$-norm.
\end{proof}

\begin{proposition} (H\"older's inequality) \label{Holder ineq}
Let $p(\cdot) \in \mathcal{P}$, then there exists a constant $C > 0$ such that
\[
\sum_{i \in \mathbb{Z}} |a(i) b(i)| dx \leq C \| a \|_{\ell^{p(\cdot)}} \| b \|_{\ell^{p'(\cdot)}}.
\]
\end{proposition}

\begin{proof} The proposition follows from \cite[Lemma 3.2.20]{Diening} considering there $A = \mathbb{Z}$ and $\mu$ being the counting measure.
\end{proof}

\begin{proposition} \label{norma equivalente}
Let $p(\cdot) \in \mathcal{P}$, then
\[
\| a \|_{\ell^{p(\cdot)}} \approx \sup \left\{ \sum_{i \in \mathbb{Z}} |a(i) b(i)| : \| b \|_{\ell^{p'(\cdot)}} \leq 1
\right\}.
\]
\end{proposition}

\begin{proof} The proposition follows from \cite[Corollary 3.2.14]{Diening}.
\end{proof}

A discrete weight is a positive real sequence. For a constant $1 < r < \infty$, we say that a weight $w = \{ w(i) \}_{i \in \mathbb{Z}}$ belongs to $\mathcal{A}_r$ if there exists a positive constant $C$ such that
\[
\left(\sum_{i=m}^{n} w(i) \right) \left(\sum_{i=m}^{n} w(i)^{-1/(r-1)} \right)^{r-1} \leq C (n-m+1)^r,
\]
for all $m, n \in \mathbb{Z}$, with $m \leq n$. For $r=1$, we say that a weight $w = \{ w(i) \}_{i \in \mathbb{Z}}$ belongs 
to $\mathcal{A}_1$ if there exists a positive constant $C$ such that
\[
\frac{1}{n-m+1} \sum_{i=m}^{n} w(i)  \leq C \inf \{ w(i) : i \in \mathbb{Z} \},
\]
for all $m, n \in \mathbb{Z}$, with $m \leq n$. Equivalently, a weight $w = \{ w(i) \}_{i \in \mathbb{Z}}$ belongs 
to $\mathcal{A}_1$ if there exists a positive constant $C$ such that
\[
(M w)(j) \leq C w(j), \,\,\,\, \text{for all} \, j \in \mathbb{Z}.
\]
Given $1 < r \leq s < \infty$, a weight $w = \{ w(i) \}_{i \in \mathbb{Z}}$ is said to belong to $\mathcal{A}_{r, s}$ if there exists a positive constant $C$ such that
\[
\left( \frac{1}{n-m+1}\sum_{i=m}^{n} w(i)^s \right)^{1/s} \left(\frac{1}{n-m+1} \sum_{i=m}^{n} w(i)^{-r'} \right)^{1/r'} \leq C,
\]
for all $m, n \in \mathbb{Z}$, with $m \leq n$.

\begin{remark} \label{Ap cond}
It is clear that $\mathcal{A}_r \subset \mathcal{A}_s$, when $1 \leq r < s < \infty$. Moreover, if $w \in \mathcal{A}_1$ then
$w^{1/s} \in  \mathcal{A}_{r,s}$ for every $1 < r \leq s < \infty$. 
\end{remark}

\section{Main results}

In this section we prove our main results by adapting the techniques used in \cite{Uribe} to our discrete setting. 

\

{\sc Proof of Theorem \ref{thm 1}.} Let $1 < r < p_{-}$ be such that $(p(\cdot)/r)' \in \mathcal{B}$. 
By \cite[Theorem 10]{Hunt} we have, for every sequence $a = \{ a(i) \}$ such that $|(Ha)(j)| < \infty$ for all $j \in \mathbb{Z}$ and every weight $w = \{ w(j) \} \in \mathcal{A}_r$, that
\begin{equation} \label{weight ineq}
\sum_{j \in \mathbb{Z}} |(Ha)(j)|^r w(j) \leq C \sum_{j \in \mathbb{Z}} |a(j)|^r w(j).
\end{equation}
In particular, since the operator $H$ is well defined on $\ell^{p(\cdot)}$, the inequality (\ref{weight ineq}) holds for 
$a \in \ell^{p(\cdot)}$. 

On the other hand, Lemma \ref{potencia r} and Proposition \ref{norma equivalente} give
\begin{equation} \label{estim Ha}
\| Ha \|_{\ell^{p(\cdot)}}^r = \| |Ha|^r \|_{\ell^{p(\cdot)/r}}  \leq C \sup \sum_{j \in \mathbb{Z}} |(Ha)(j)|^r b(j),
\end{equation}
where the supremum is taken over all non-negative sequences $b = \{ b(j) \} \in \ell^{(p(\cdot)/r)'}$ with 
$\| b \|_{\ell^{(p(\cdot)/r)'}} \leq 1$.
Since, by hypothesis, the discrete maximal $M$ is bounded on $\ell^{(p(\cdot)/r)'}$, we define the operator $\mathcal{R}$ on 
$\ell^{(p(\cdot)/r)'}$ by
\[
(\mathcal{R}b)(j) = \sum_{k=0}^{\infty} \frac{(M^k b)(j)}{2^k A},
\]
where $A = \|M \|_{\ell^{(p(\cdot)/r)'} \to \ell^{(p(\cdot)/r)'}}$, for $k \geq 1$, $M^k = M \circ \cdot \cdot \cdot \circ M$ denotes $k$ iterations of the discrete maximal, and $M^0$ is the identity operator. It is easy to check that

$(i)$ if $b(j) \geq 0$, then $b(j) \leq (\mathcal{R}b)(j)$ for all $j \in \mathbb{Z}$;

$(ii)$ $\| \mathcal{R}b \|_{\ell^{(p(\cdot)/r)'}} \leq 2 \| b \|_{\ell^{(p(\cdot)/r)'}}$;

$(iii)$ $M(\mathcal{R}b)(j) \leq 2A (\mathcal{R}b)(j)$, for all $j \in \mathbb{Z}$, so $\mathcal{R}b \in \mathcal{A}_1$. \\
Now, from $(i)$, it follows that
\[
\sum_{j \in \mathbb{Z}} |(Ha)(j)|^r b(j) \leq \sum_{j \in \mathbb{Z}} |(Ha)(j)|^r (\mathcal{R}b)(j),
\]
since $\mathcal{R}b \in \mathcal{A}_1$ (see $(iii)$), by Remark \ref{Ap cond} and (\ref{weight ineq}), we obtain
\[
\leq \sum_{j \in \mathbb{Z}} |a(j)|^r (\mathcal{R}b)(j),
\]
to apply Proposition \ref{Holder ineq} on this expression, we get 
\[
\leq C \| |a|^r \|_{\ell^{p(\cdot)/r}} \| \mathcal{R}b \|_{\ell^{(p(\cdot)/r)'}},
\]
the property $(ii)$ gives
\[
\leq C \| |a|^r \|_{\ell^{p(\cdot)/r}} \| b \|_{\ell^{(p(\cdot)/r)'}}.
\]
The above string of inequalities leads to
\begin{equation} \label{estim Ha2}
\sum_{j \in \mathbb{Z}} |(Ha)(j)|^r b(j) \leq C \| |a|^r \|_{\ell^{p(\cdot)/r}} \| b \|_{\ell^{(p(\cdot)/r)'}}.
\end{equation}
Finally, by taking the supremum over all non-negative sequences $b = \{ b(j) \} \in \ell^{(p(\cdot)/r)'}$ such that 
$\| b \|_{\ell^{(p(\cdot)/r)'}} \leq 1$ in (\ref{estim Ha2}), from (\ref{estim Ha}) and Lemma \ref{potencia r}, we conclude that
\[
\| Ha \|_{\ell^{p(\cdot)}} \leq C \| a \|_{\ell^{p(\cdot)}}.
\]
This completes the proof. \hspace{11.8cm} $\square$

\

{\sc Proof of Theorem \ref{thm 2}.} Let $\frac{1}{1-\alpha} < s < q_{-}$ be such that $(q(\cdot)/s)' \in \mathcal{B}$. We define $\frac{1}{r} = \frac{1}{s} + \alpha$, so $1 < r < p_{-}$. Since the operator $I_{\alpha}$ is well defined on $\ell^{p(\cdot)}$, by
Remark \ref{Ap cond} and \cite[Theorem 4.4]{Hao2}, the inequality
\begin{equation} \label{Ia weight}
\sum_{j \in \mathbb{Z}} |(I_{\alpha}a)(j)|^s w(j) \leq C \left( \sum_{j \in \mathbb{Z}} |a(j)|^r w(j)^{r/s} \right)^{s/r}
\end{equation}
holds for all $a =\{ a(i) \} \in \ell^{p(\cdot)}$ and all $w = \{ w(i) \} \in \mathcal{A}_1$. By Lemma \ref{potencia r} and 
Proposition \ref{norma equivalente}, we have
\begin{equation} \label{estim Ia}
\| I_{\alpha}a \|_{\ell^{q(\cdot)}}^s = \| |I_{\alpha}a|^s \|_{\ell^{q(\cdot)/s}}  \leq 
C \sup \sum_{j \in \mathbb{Z}} |(I_{\alpha}a)(j)|^s b(j),
\end{equation}
where the supremum is taken over all non-negative sequences $b = \{ b(j) \} \in \ell^{(q(\cdot)/s)'}$ with 
$\| b \|_{\ell^{(q(\cdot)/s)'}} \leq 1$. By our assumption on $q(\cdot)$, the discrete maximal $M$ is bounded on 
$\ell^{(q(\cdot)/s)'}$. Then, we define the operator $\mathcal{R}$ on $\ell^{(q(\cdot)/s)'}$ by
\[
(\mathcal{R}b)(j) = \sum_{k=0}^{\infty} \frac{(M^k b)(j)}{2^k A},
\]
where $A = \|M \|_{\ell^{(q(\cdot)/s)'} \to \ell^{(q(\cdot)/s)'}}$. We observe that the properties $(i)$ and $(iii)$ above hold and $(ii)$ holds with $q(\cdot)/s$ instead of $p(\cdot)/r$. Proceeding as in the proof of Theorem \ref{thm 1}, but now considering (\ref{Ia weight}) instead of (\ref{weight ineq}), we obtain
\begin{equation} \label{estim Ia2}
\sum_{j \in \mathbb{Z}} |(I_{\alpha}a)(j)|^s b(j) \leq C \| |a|^r \|_{\ell^{p(\cdot)/r}}^{s/r} 
\| (\mathcal{R}b)^{r/s} \|_{\ell^{(p(\cdot)/r)'}}^{s/r} .
\end{equation}
A computation gives $\| (\mathcal{R}b)^{r/s} \|_{\ell^{(p(\cdot)/r)'}}^{s/r} = \| \mathcal{R}b \|_{\ell^{(q(\cdot)/s)'}} \leq C 
\| b \|_{\ell^{(q(\cdot)/s)'}}$. So, the theorem follows from this inequality, (\ref{estim Ia2}) and (\ref{estim Ia}).
\hspace{10.5cm} $\square$

\

{\sc Proof of Theorem \ref{thm 3}.} We first consider the case $0 < \alpha < 1$. We define
\[
\mathcal{F}_{\alpha} = \left\{ \left( \left\{ \sum_{k=1}^{N} (M_{\alpha}a_k)^{\theta} \right\}^{1/\theta}, 
\left\{ \sum_{k=1}^{N} |a_k|^{\theta} \right\}^{1/\theta} \right) : N \in \mathbb{N}, 
\{a_k \}_{k=1}^{N} \subset \ell^{\infty}_{comp} \right\},
\]
where $\theta \in (1, \infty)$ and $\ell^{\infty}_{comp}$ denotes the set of sequences with finite support on $\mathbb{Z}$.

Let $\frac{1}{1-\alpha} < s < q_{-}$ be such that $(q(\cdot)/s)' \in \mathcal{B}$, and let $r$ be defined by $\frac{1}{r} = \frac{1}{s} + \alpha$. So, $1 < r < p_{-}$. By \cite[Theorem 3.3]{Hao2}, \cite[Theorem 3.23]{Cruz} (which also holds in the discrete setting), and Remark \ref{Ap cond}, there exists an universal constant $C > 0$ such that for any $(F, G) \in \mathcal{F}_{\alpha}$ and any $w=\{ w(i)\} \in \mathcal{A}_1$ 
\begin{equation} \label{FG weight}
\sum_{j \in \mathbb{Z}} [F(j)]^{s} w(j)  \leq C \left( \sum_{j \in \mathbb{Z}} [G(j)]^{r} w(j)^{r/s}  \right)^{s/r}.
\end{equation}
Now, proceeding as in the proof of Theorem \ref{thm 2}, we apply the Rubio de Francia iteration algorithm with respect to 
$\ell^{(q(\cdot)/s)'}$ to obtain
\[
\| F \|_{\ell^{q(\cdot)}} \leq C \| G \|_{\ell^{p(\cdot)}},
\]
for all pair $(F, G) \in \mathcal{F}_{\alpha}$. This is
\begin{equation} \label{fract max ineq2}
\left\| \left\{\sum_{k=1}^{N} (M_{\alpha}a_k)^{\theta} \right\}^{1/\theta} \right\|_{\ell^{q(\cdot)}} \leq 
C \left\| \left\{\sum_{k=1}^{N} |a_k|^{\theta} \right\}^{1/\theta}  \right\|_{\ell^{p(\cdot)}}.
\end{equation}
Then, letting $N \to \infty$ in (\ref{fract max ineq2}), we obtain (\ref{fract max ineq}) for the case $0 < \alpha < 1$. 

Finally, by \cite[Proposition 2.9]{Hao} and Remark \ref{Ap cond}, we have that (\ref{FG weight}) holds with $\alpha=0$ and $r=s$. Thus, by reasoning as above, the theorem follows for the case $\alpha = 0$. \hspace{2.8cm} $\square$

\bigskip
\address{
Departamento de Matem\'atica \\
Universidad Nacional del Sur (UNS) \\
Bah\'{\i}a Blanca, Argentina}
{pablo.rocha@uns.edu.ar}


\begin{thebibliography}{99}

\bibitem{Uribe} D. Cruz-Uribe, A. Fiorenza, J. M. Martell and C. P\'erez, {\it The boundedness of classical operators on variable $L^p$ spaces}, Ann. Acad. Sci. Fenn., Math. 31, No. 1 (2006), 239-264.

\bibitem{Cruz} D. V. Cruz-Uribe, J. M. Martell and C. P\'erez, Weights, Extrapolation and the Theory of Rubio de Francia, Birkh\"auser. Operator theory: advances and applications, Vol. 215 (2011).

\bibitem{Diening} L. Diening, P. Harjulehto, P. H\"ast\"o and M. R\r u$\check{\text{z}}$i$\check{\text{c}}$ka, Lebesgue and Sobolev spaces with Variable Exponents, Springer, 2011.  

\bibitem{Edmunds} D. E. Edmunds and A. Nekvinda, {\it Averaging operators on on $\ell^{\{p_n \}}$ and $L^{p(x)}$}, Math. inequal. Appl. 
5 (2), (2002), 235-246.

\bibitem{Hao} X. Hao, B. Li and S. Yang, {\it The Hardy-Littlewood maximal operator on discrete weighted Morrey spaces}, Acta Math. Hung. 
172, (2) (2024), 445-469.

\bibitem{Hao2} X. Hao, B. Li and S. Yang, {\it Estimates of Discrete Riesz Potentials on Discrete Weighted Lebesgue Spaces}, Ann. Funct. Anal. 15 (3), Paper No. 51 (2024), 24 p.

\bibitem{Hardy} G. H. Hardy, J. E. Littlewood and G. P\'olya, Inequalities, 2nd ed., Cambridge Univ. Press, London and new York, 1952.

\bibitem{Hasto} P. H\"ast\"o, {\it Local-to-global results in variable exponent spaces}, Math. Res. Letters, 16 (2) (2009), 263-278. 

\bibitem{Hunt} R. Hunt, B. Muckenhoupt and R. Wheeden, {\it Weighted norm inequalities for the conjugate function and Hilbert transform},
Trans. Amer. Math. Soc. 176 (1973), 227-251. 

\bibitem{Nakano} H. Nakano, {\it Modulared sequence spaces}, Proc. Japan Acad., 27 (1951), 508-512.

\bibitem{Nekvinda} A. Nekvinda, {\it Equivalence of $\ell^{\{p_n \}}$ norms and shift operators}, Math. Inequal. Appl. 5 (4), (2002), 711-723.

\bibitem{Nekvinda2} A. Nekvinda, {\it A note on maximal operator on $\ell^{\{p_n \}}$ and $L^{p(x)}(\mathbb{R})$}, J. Funct. Spaces 
Appl. 5 (1), (2007), 49-88.

\bibitem{Orlicz} W. Orlicz, {\it \"Uber konjugierte Exponentenfolgen}, Studia Math., 3 (1931), 200-211. 

\bibitem{Riesz} M. Riesz, {\it Sur les fonctions conjugu\'ees}, Math. Z., 27 (1928), 218-244.

\bibitem{Rocha} P. Rocha, {\it A note about discrete Riesz potential on $\mathbb{Z}^n$}, Preprint 2024: https://arxiv.org/abs/2407.15262

\bibitem{Swarup} A. S. S. Swarup and A. M. Alphonse, {\it The boundedness of Fractional Hardy-Littlewood maximal operator on variable $\ell^{p(\cdot)}(\mathbb{Z})$ spaces using Calder\'on-Zygmund decomposition}. J. Indian Math. Soc., New Ser. 91, No. 1-2 (2024), 237-252.


\end{thebibliography}
\end{document}